\documentclass[12pt,reqno]{amsart}

\usepackage{amsfonts, amsthm, amsmath, amssymb}

\usepackage{rotating}
\usepackage{tikz}
\usetikzlibrary{shapes}
\usetikzlibrary{matrix,arrows}
\usetikzlibrary{positioning}
\usetikzlibrary{fit}
\usetikzlibrary{patterns}

\usepackage{amscd}

\usepackage[latin2]{inputenc}

\usepackage{t1enc}

\usepackage[mathscr]{eucal}

\usepackage{indentfirst}

\usepackage{graphicx}

\usepackage{graphics}

\usepackage{pict2e}

\usepackage{mathrsfs}

\usepackage{enumerate}
\usepackage[pagebackref]{hyperref}
\usepackage{cite}
\usepackage{color}
\usepackage{epic}
\usepackage{hyperref} 
\numberwithin{equation}{section}
\def\red{\textcolor{red}}
\theoremstyle{plain}

\tikzstyle{pathdefault}=[draw, line width=1, solid, color=black]
\tikzstyle{nodedefault}=[circle, inner sep=1.5, fill=black]
\tikzstyle{empty}=[]
\tikzstyle{nodeellipsis}=[circle, inner sep=0.5, fill=black]
\tikzstyle{pathcolor1}=[draw, line width=1.3, densely dashed, color=red]
\tikzstyle{pathcolor2}=[draw, line width=1.6, densely dotted, color=blue]
\tikzstyle{pathcolorlight}=[draw, line width=1, dotted, color=lightgray]

\tikzstyle{arbpathcolor0}=[line width=1, dashdotted, color=black]
\tikzstyle{arbpathcolor1}=[line width=1, densely dashed, color=red]
\tikzstyle{arbpathdefault}=[line width=1, densely dotted, color=blue]

\newcounter{id}
\newcommand{\drawlinedotswithstyle}[4]{
 \def\x{{#3}}
 \def\y{{#4}}
 \tikzstyle{thispathstyle}=[#1]
 \tikzstyle{thisnodestyle}=[#2]
 \setcounter{id}{-1} 
 \foreach \j in {#3}{\stepcounter{id}} 
 \foreach \i in {1,...,\the\value{id}}{  
  \path[thispathstyle] (\x[\i],\y[\i]) --(\x[\i-1],\y[\i-1]); 
 }
 \foreach \i in {1,...,\the\value{id}}{  
  \node[thisnodestyle] at (\x[\i],\y[\i]) {}; 
 }
 \node[thisnodestyle] at (\x[0],\y[0]) {}; 
}

\definecolor{mhcblue}{HTML}{0077CC} 
\definecolor{davidsonred}{HTML}{AC1A2F} 

\definecolor{green}{RGB}{0, 180, 0}
\definecolor{yellow}{RGB}{180, 180, 0}

\newtheorem{theorem}{Theorem}[section]

\newtheorem{lemma}[theorem]{Lemma}

\newtheorem{proposition}[theorem]{Proposition}

\theoremstyle{definition}

\newtheorem{Def}[theorem]{Definition}

\newtheorem{remark}[theorem]{Remark}

\newtheorem{?}[theorem]{Problem}

\newcommand{\Z}{\mathbb{Z}}

\def\min{\mathrm{min}}

\def\dd{\mathrm{dd}}
\def\des{\mathrm{des}}

\def\S{\mathfrak{S}}

\def\da{\mathrm{da}}

\def\valley{\mathrm{valley}}
\def\val{\mathrm{val}}

\def\pk{\mathrm{pk}}

\def\Orb{\mathrm{Orb}}

\def\bac{2\text{-}13}
\def\cab{31\text{-}2}
\def\bca{2\text{-}31}
\def\acb{13\text{-}2}

\def\cabd{31\text{-}2\text{-}4}
\def\cadb{31\text{-}4\text{-}2}
\def\dabc{41\text{-}2\text{-}3}
\def\dacb{41\text{-}3\text{-}2}

\newcommand{\bB}{{\mathcal B}}

\newcommand\vartextvisiblespace[1][.5em]{%
  \makebox[#1]{%
    \kern.07em
    \vrule height.3ex
    \hrulefill
    \vrule height.3ex
    \kern.07em
  }%
}

\newcommand{\x}{\vartextvisiblespace}

\def\lval{\mathrm{lval}}
\def\lpk{\mathrm{lpk}}
\def\lda{\mathrm{lda}}
\def\ldd{\mathrm{ldd}}

\def\boxit#1{\leavevmode\hbox{\vrule\vtop{\vbox{\kern.33333pt\hrule
    \kern1pt\hbox{\kern1pt\vbox{#1}\kern1pt}}\kern1pt\hrule}\vrule}}

\usepackage{collectbox}



\begin{document}

\title[gamma expansion and pattern avoidance]{$(p,q,t)$-Catalan continued fractions, gamma expansions and pattern avoidances}

\author[B. Han]{Bin Han}
\address[Bin Han]{Department of Mathematics, Royal institute of Technology (KTH), SE 100-44 Stockholm, Sweden} 
\email{binhan@kth.se, han.combin@hotmail.com}
\author[Q. Pan]{Qiongqiong Pan}
\address[Qiongqiong Pan]{College of mathematics and physics, Wenzhou University, Wenzhou 325035, PR China}
\email{qpan@wzu.edu.cn}
%
%
%
  
\date{\today}

\begin{abstract} 
We introduce a kind of $(p, q, t)$-Catalan numbers of Type A by generalizing the Jacobian type continued fraction formula, we prove that the corresponding expansions could be expressed by the polynomials counting permutations on $\S_n(321)$ by various descent statistics. 
Moreover, we introduce a kind of $(p, q, t)$-Catalan numbers of Type B by generalizing the Jacobian type continued fraction formula, we prove that the Taylor coefficients and their $\gamma$-coefficients could be expressed by the polynomials counting permutations on $\S_n(3124, 4123, 3142, 4132)$ by various descent statistics. 
Our methods include permutation enumeration techniques involving variations of bijections from permutation patterns to labeled Motzkin paths and modified Foata-Strehl action.

\end{abstract}
\keywords{Catalan numbers; pattern avoidance; $\gamma$-positive; continued fraction; group action}

\maketitle

\tableofcontents

\section{Introduction}
It is well-known that \emph{Catalan numbers} $C_n=\frac{1}{n+1}{2n\choose n}$ has the following Jacobi-type continued fraction expansion (cf. \cite{BCS08,Fla80,FTHZ})
\begin{align*}
\sum_{n\geq 0} C_{n} z^n=\cfrac{1}{1-1\cdot z-\cfrac{1\cdot z^{2}}{1-2\cdot z-\cfrac{
 1\cdot z^{2}}{1-\cdots}}}
\end{align*}
 In this paper we define the following {\em (p,\,q,\,t)-Catalan continued fraction of Type A} as
\begin{align}\label{Pqcatalan}
&\sum_{n\geq0}
C_n(p,\,q,\,t)z^n=\nonumber\\
&\hspace{0.5cm}\cfrac{1}{1-1\cdot z-\cfrac{tp\cdot z^{2}}{1-(p+q)\cdot z-\cfrac{
 tp^2q\cdot z^{2}}{1-(p^2+q^2)\cdot z-\cfrac{tp^3q^2\cdot z^2}{1-(p^3+q^3)\cdot z-\cfrac{tp^4q^3\cdot z^2}{1-\cdots}}}}}
\end{align}
Where $C_n(p,\,q,\,t)$ is the  {\em $(p,\,q,\,t)$-Catalan  numbers of type A} as the Taylor coefficients in the above continued fraction.

In particular, let $\bar{C}_n(q):=C_n(1,\,q,\,1)$ and $\widetilde{C}_n(q):=C_n(q,\,1,\,1)$. 
We need a standard \emph{contraction formula} for continued fractions, see \cite[Lemma~5.1]{HMZ20}.
\begin{lemma}[Contraction formula]\label{contra-formula}
The following holds
\begin{align*}
\cfrac{1}{
1-\cfrac{c_{1}z}{1-\cfrac{c_{2}z}{
1-\cfrac{c_{3}z}{
1-\cfrac{c_{4}z}
{\ddots}}}}}
&=\cfrac{1}{1-c_{1}z-\cfrac{c_{1}c_{2}z^{2}}{1-(c_{2}+c_{3})z-\cfrac{c_{3}c_{4}z^{2}}{1-(c_4+c_5)z-\cfrac{c_5c_6z^2}{\ddots}}}}.\\
\end{align*}
\end{lemma}
By contraction we derive the two $q$-Catalan numbers
\begin{equation}\label{Adef-q-cat1}
\sum_{n=0}^\infty \bar{C}_n(q) z^n=
\cfrac{1}{
1-\cfrac{z}{1-\cfrac{z}{
\cfrac{\ddots}{1-\cfrac{q^{k-1}z}{1-\cfrac{z}{\ddots}
}}}}},
\end{equation}

and 

\begin{equation}\label{Adef-q-cat2}
\sum_{n=0}^\infty \widetilde{C}_n(q) z^n=
\cfrac{1}{
1-\cfrac{z}{1-\cfrac{qz}{
\cfrac{\ddots}{1-\cfrac{z}{1-\cfrac{q^kz}{\ddots}
}}}}}.
\end{equation}

\begin{figure}[t]\label{fig1}
  $$
  \begin{array}{c|cccc|c}
\hbox{$n$}\backslash\hbox{$k$}&0&1&2&3&C_n\\
\hline
1& 1&&&&1\\
2& 2&&&&2\\
3& 4&1&&&5\\
4& 8&5&1&&14\\
5& 16&17&7&2&42\\
\end{array}
\qquad\qquad
 \begin{array}{c|cccccc|c}
 n\diagdown  k &0&1&2&3&4&5&C_n\\
 \hline
 1&1&&&&&&1\\
 2&1&1&&&&&2\\
 3&1&3&1&&&&5\\
 4&1&6&5&2&&&14\\
 5&1&10&15&12&3&1&42\\
 \end{array}
 $$
 \caption{The first values of $\bar{C}_n(q)$ (left) and  $\widetilde{C}_n(q)$ (right) for $0< n\leq 5$.}
 \end{figure}

We give the first few terms for $\bar{C}_n(q)$ and $\widetilde{C}_n(q)$ see Fig~\ref{fig1} and note that neither $\bar{C}_n(q)$ nor $\widetilde{C}_n(q)$ is registered in the OEIS.


\emph{Type B Catalan numbers} $\bB_n={2n\choose n}$ has  the following Jacobi-type continued fraction expansion (cf. \cite{Fla80, Sok19})
\begin{align}\label{2-motzkin}
\sum_{n\geq 0} \bB_{n} z^n=\cfrac{1}{1-2\cdot z-\cfrac{2\cdot z^{2}}{1-2\cdot z-\cfrac{
 1\cdot z^{2}}{1-\cdots}}}.
\end{align}

We define the following {\em (p,\,q,\,t)-Catalan continued fraction of Type B} as
\begin{align}
\sum_{n=0}^{\infty}\bB_n(p, q,t)z^{n}=\cfrac{1}{1-(1+t)z-\cfrac{(p+q)tz^2}{1-p(1+t)z-\cfrac{p^3tz^2}{1-p^2(1+t)z-\cfrac{p^5tz^2}{\ddots}}}}.
\end{align}
Where $\bB_n(p,\,q,\,t)$ is the  {\em $(p,\,q,\,t)$-Catalan  numbers of Type B} as the Taylor coefficients in the above continued fraction.

The combinatorial constructions behind the proof of Theorem~\ref{Apqcatalan} and Theorem~\ref{Bpqcatalan}, two of our main results, originated from the fundamental work of Flajolet~\cite{Fla80} for the lattice path interpretation of the formal continued fractions, and a bijection between sets of certain weighted  \emph{Motzkin paths} and permutations  due to Fran\c con-Viennot\cite{FV}, see also \cite{CSZ, FZ, SZ10,SZ12}, and the {\em insertion encoding}, an encoding of finite permutations introduced by Albert, Linton and Ru\v{s}kuc in~\cite{ALR05}.

The goal of this paper is to establish combinatorial interpretations for $C_n(p, q,t)$ and
$\bB_n(p, q,t)$ as well as their corresponding $\gamma$-expansions, using pattern avoiding permutations, which we define now.
Denote by $\S_n$ the set of permutations of length $n$. Given two permutations $\pi\in\S_n$ and $p\in\S_k,\:k\leq n$, we say that \emph{$\pi$ avoids the pattern $p$} if there does not exist a set of indices $1\le i_1<i_2<\cdots<i_k\le n$  such that the subsequence $\pi(i_1)\pi(i_2)\cdots\pi(i_k)$ of $\pi$ is order-isomorphic to $p$. For example, the permutation $15324$ avoids $231$. The set of permutations of length $n$ that avoid patterns $p_1,p_2,\cdots,p_m$ is denoted as $\S_n(p_1,p_2,\cdots,p_m)$.


 Many polynomials with combinatorial meanings have been shown to be unimodal; 
see the recent survey of Br\"and\'en~\cite{bran}. Recall that a polynomial $h(t)=\sum_{i=0}^dh_it^i$ of degree $d$ is said to be {\em unimodal} if the coefficients are increasing 
and then decreasing, i.e., there is an index $c$ such that
$
h_0\leq h_1\leq \cdots\leq h_c\geq h_{c+1}\geq\cdots\geq h_d.
$
Let $p(t)=a_rt^r+a_{r+1}t^{r+1}+\cdots +a_{s} t^{s}$ be a real polynomial with $a_r\neq0$ and $a_s\neq0$. It is called  \emph{palindromic}   (or \emph{symmetric}) {\em of center $n/2$}  if $n=r+s$ and $a_{r+i}=a_{s-i}$ for $0\leq i\leq n/2$. For example,
 polynomials $1+t$ and $t$ are palindromic of center  $1/2$ and 1, respectively.
Any palindromic  polynomial $p(t)\in \Z[t]$  can be written uniquely~\cite{bran,swz} as
  $$
p(t)=\sum_{k=r}^{\lfloor\frac{n}{2}\rfloor}\gamma_{k}t^k(1+t)^{n-2k},
$$
where $\gamma_{k}\in\Z$. If $\gamma_{k}\geq 0$ then we say that it is 
 {\em $\gamma$-positive of  center $n/2$}.
It is clear  that the 
$\gamma$-positivity implies palindromicity and unimodality. 
For further $\gamma$-positivity results and problems, the reader is referred to the excellent exposition by Petersen~\cite{Pet} and  the most recent survey
 by Athanasiadis~\cite{Ath}.

Now we give the first three main results of this paper, with the definitions of permutation statistics, permutation patterns
postponed to the next section.

\begin{theorem}\label{Apqcatalan}
The $(p,q,t)$-Type A Catalan numbers $C_n(p, q,t)$ have the following interpretation
\begin{align*}
C_n(p, q,t)=\sum_{\sigma\in\S_n(321)}p^{\widehat{(\bac)}\:\sigma}q^{(\cab)\,\sigma}t^{\des\:\sigma}.
\end{align*}

\end{theorem}
\begin{remark}
 Barnabei, Bonetti and Silimbani \cite[Theorem~6]{BBS10} gave an ordinary generating function formula for $t^nC_n(1, 1,1/t)$ by exploiting Krattenthaler's bijection between 123-avoiding permutations and Dyck paths(see OEIS: A166073).

\end{remark}

\begin{theorem}\label{Bpqcatalan}
The $(p,q,t)$-Type B Catalan numbers $\bB_n(p, q, t)$ have the following interpretation
\begin{align*}
\bB_n(p, q, t)=\sum_{\sigma\in\S_{n+1}(3124, 4123, 3142, 4132)}p^{(\bac)\:\sigma}q^{(\cab)\,\sigma}t^{\des\:\sigma}.
\end{align*}

\end{theorem}
\begin{remark}
When $p=q=t=1$, Albert, Linton and Ru\v{s}kuc in \cite[Section~7.4]{ALR05} gave the above combinatorial interpretation for Type B Catalan numbers. 
\end{remark}

\begin{theorem}\label{thm:des-qgamma}
For  $n\geq1$, the following $\gamma$-expansions formula holds true
\begin{align}\label{qgamma-unify}
\bB_n(p, q,t)=\sum_{k=0}^{\lfloor \frac{n}{2}\rfloor}
\gamma_{n+1,k}(p, q)t^k(1+t)^{n-2k},
\end{align}
where
\begin{equation}\label{qgammacoe}
\gamma_{n+1, k}(p, q):= \sum_{\sigma\in \S_{n+1, k}}p^{(\bac)\,\sigma}q^{(\cab)\,\sigma},
\end{equation}
and
\begin{align*}
\S_{n+1, k}:=\{\sigma\in \S_{n+1}(3124, 4123, 3142, 4132): \dd(\sigma)=0, \val\,\sigma=\des\,\sigma=k\}.
\end{align*}
\end{theorem}

For example, the first few expansions  of $\bB_n(p, q,t)$  read as follows:
\begin{align*}
\bB_1(p, q,t)=&1+t;\\
\bB_2(p, q,t)=&(1+t)^2+(p+q)t;\\
\bB_3(p, q,t)=&(1+t)^3+(p+2)(p+q)t(1+t);\\
\bB_4(p, q,t)=&(1+t)^4+(p+q)(p^2+2p+3)t(1+t)^2+t^2(p^3+p+q)(p+q);\\
\bB_5(p, q,t)=&(1+t)^5+(p+q)(p^3+2p^2+3p+4)t(1+t)^3\\
&+t^2(1+t)(p+q)(p^5+2p^4+2p^3+2p^2+(2q+3)p+3q).
\end{align*}

\begin{remark}
Postnikov, Reiner and Williams \cite[Proposition~11.15]{prw} showed $$\gamma_{n+1,k}(1,1)={{n}\choose{k,k, n-2k}}$$ by the corresponding $h$-polynomials and standard quadratic transformations of hypergeometric series.  However, we give $\gamma_{n,k}(p,q)$ by continued fraction theory and bijection between permutation patterns and weighted Motzkin paths.


\end{remark}

The rest of this paper is organized as follows. In Section~\ref{sec2: Pre}, we give most of the definitions and provide the previously known results, which will be used to prove Theorems~\ref{Apqcatalan} and \ref{Bpqcatalan} in Section~\ref{Sec:proofA} and Section~\ref{Sec:proofB} respectively. In Section~\ref{Sec:proofg}, we introduce Br\"and\'en's modified Foata-Strehl action and prove the theorem~\ref{thm:des-qgamma}. We consider a variation of $(q,t)$-Catalan numbers in Section~\ref{Varicatalan}.

\section{Definitions and Preliminaries}\label{sec2: Pre}

\subsection{Permutation statistics}
For $\sigma=\sigma(1)\sigma(2)\cdots\sigma(n)\in\S_n$ with convention $\sigma(0)=\sigma(n+1)=0$, a value $\sigma(i)$ ($1\leq i\leq n$) is called
\begin{itemize}
\item 
a \emph{peak} if $\sigma(i-1)<\sigma(i)$ and $\sigma(i)>\sigma(i+1)$;
\item
a \emph{valley} if $\sigma(i-1)>\sigma(i)$ and $\sigma(i)<\sigma(i+1)$;
\item
a \emph{double ascent} if $\sigma(i-1)<\sigma(i)$ and $\sigma(i)<\sigma(i+1)$;
\item
a \emph{double descent} if $\sigma(i-1)>\sigma(i)$ and $\sigma(i)>\sigma(i+1)$.
\end{itemize}
And we call a permutation $\sigma$ has a descent at $i$ with $1\leq i<n$ if $\sigma(i)>\sigma(i+1)$. 
The set of peaks (resp. valleys, double ascents, double descents, descents) of $\sigma$
is denoted by 
$$\mathrm{Pk}\, \sigma\,\,(\text{resp.}\; \mathrm{Val}\,\sigma, \,\mathrm{Da}\, \sigma,\, \mathrm{Dd}\, \sigma,\,\mathrm{Des}\,\sigma).
$$
Let $\pk\,\sigma$ (resp. $\val\,\sigma,\,\da\,\sigma,\,\dd\,\sigma,\,\des\,\sigma$) be the number of peaks (resp. valleys, double ascents, double descents, descents) of $\sigma$. For $i\in[n]:=\{1,\ldots,n\}$, we introduce the following statistics:
\begin{align}
(\cab)_i\,\sigma=\#\{j:1<j<i\; \text{and}\; \sigma(j)<\sigma(i)<\sigma(j-1)\}\\
(\bca)_i\,\sigma=\#\{j:i<j<n\; \text{and} \;\sigma(j+1)<\sigma(i)<\sigma(j)\}\\
(\bac)_i\,\sigma=\#\{j:i<j<n \;\text{and} \;\sigma(j)<\sigma(i)<\sigma(j+1)\}\\
(\acb)_i\,\sigma=\#\{j:1<j<i \;\text{and}\; \sigma(j-1)<\sigma(i)<\sigma(j)\}
\end{align}
and define four statistics:
$$
(\cab)=\sum_{i=1}^n(\cab)_i,\hspace{0.4cm} (\bca)=\sum_{i=1}^n(\bca)_i,\hspace{0.4cm} (\bac)=\sum_{i=1}^n(\bac)_i,\hspace{0.4cm} (\acb)=\sum_{i=1}^n(\acb)_i.
$$
\\
For $\sigma\in\S_n$ with convention $\sigma(0)=0$ and $\sigma(n+1)=\infty$, let
$$
\mathrm{Lpk}\, \sigma\,\,(\text{resp.}\; \mathrm{Lval}\, \sigma, \,\mathrm{Lda}\, \sigma,\, \mathrm{Ldd}\,\sigma)
$$
be the set of peaks (resp. valleys, double ascents and double descents) of $\sigma$ and denote the corresponding cardinality by $\lpk$ (resp. $\lval$, $\lda$, $\ldd$).

\subsection{Insertion encoding of permutations}
Every permutation can be generated from the empty permutation by successive insertion of a new maximum element. We call this procedure the \emph{evolution} of a permutation. This 
observation is used as a basis for the generating tree methodology. In this methodology, a partially constructed permutation belonging to some specified pattern class is viewed as having a number of active sites-points at which a new maximum could be inserted while remaining in the pattern class. In~\cite{ALR05}, Albert, Linton and Ru\v{s}kuc modified this viewpoint slightly in that for them an active site in the construction of a permutation is one in which a new maximum element will eventually be inserted.
\vskip 1mm
At each preliminary stage we index the slots which are present from left to right beginning with $1$. Then each insertion is determined by the following:
\begin{itemize}
\item 
The index of the slot in which it occurs.
\item
The way in which the slot is affected by the insertion.
\end{itemize}
 There are four different types of insertion of a new maximum element $n$ within a slot:
 \begin{itemize}
\item 
$\x\longrightarrow\x\, n\,\x$ represented by $m$ (for middle)
\item
$\x\longrightarrow n\,\x$ represented by $l$ (for left)
\item
$\x\longrightarrow \x\,n$ represented by $r$ (for right)
\item
$\x\longrightarrow n$ represented by $f$ (for fill)
\end{itemize}
If we subscript the insertion type with the slot on which it operates, we obtain a uniquely defined encoding of any permutation. For instance, the insertion encoding of $423615$ is $m_1m_1l_2f_1f_2f_1$.
\\
For clarity, we sometimes use  negative subscripts in counting slots from the right. That is $m_{-1}$ represents a middle insertion into the rightmost slot and $r_{-2}$ represents an insertion on the right hand end of the second slot from the right. There are also situations where \emph{a modification of the insertion encoding} is desirable. Suppose that we are interested in studying a collection of permutations $\mathcal{C}$ with the property that for any $\pi,\,\tau\in\mathcal{C}$ the permutation $\pi\oplus\tau$ which has an initial segment of small values order isomorphic to $\pi$ followed by a segment of larger values order isomorphic to $\tau$ is also in $\mathcal{C}$ (for instance,$21\oplus312=21534$). Then it is natural to retain a "free slot" on the right hand end of any configuration. We enforce this by simply forbidding the $f$ and $r$ operations in the final slot.

\vskip 1mm
In~\cite{ALR05}, the authors described several pattern classes by insertion encoding and gave the following properties.
\begin{proposition}\cite[Proposition~4]{ALR05}\label{P1}
The permutations in $\S_n(312)$ are precisely those whose insertion encoding uses only the symbols $f_1$, $l_1$, $r_1$, and $m_1$.
\end{proposition}
In~ \cite{ALR05}, the authors used the modification of the insertion encoding to give grammar describes for $\S_n(321)$. Here, we re-describe $\S_n(321)$ in the flavor of Proposition~\ref{P1}.

\begin{proposition}\label{P2}
The permutations in $\S_n(321)$ are precisely those whose insertion encoding uses only the symbols $m_{-1}$ and $l_{-1}$ if there is only one free slot, and uses only the symbols $f_1$, $l_1$, $l_{-1}$ and $m_{-1}$ if there are more than one free slot.
\end{proposition}

\begin{proof}
In this modified form, it is obviously that the only letters which may occur when only one slot is free are $m_{-1}$ and $l_{-1}$. If there are more than one free slot, insert $f_1$, $l_1$, $l_{-1}$ and $m_{-1}$ will not create a 321 pattern and also keep a free slot at the end of sequence. But, if we insert other operations, either they eliminate the free slot at the end of the figure or they will create a 321 pattern since there is a smaller element at the right side of this operation and there have at least one free slot, once an element is added to the first slot of this sequence, a $321$ pattern is created.
\end{proof}

\begin{proposition}\cite[Section~7.4]{ALR05}\label{P3}
The permutations in $\S_n(3124,4123,3142,4132)$ are precisely those whose insertion encoding uses only the symbols $m_1$, $l_1$, $r_1$ and $f_1$, but if there are two free slots then the insertion encoding can also use $f_2$.
\end{proposition}

%


\subsection{Laguerre histories as permutation encodings}

%
A \emph{Laguerre history} (resp. \emph{restriced Laguerre history}) of length $n$ is a pair of $(\mathbf{s},\mathbf{p})$, where $\mathbf{s}$ is a 2-Motzkin path $s_1\ldots s_n$ and $\mathbf{p}=(p_1,\ldots,p_n)$ with $0\leq p_i\leq h_{i-1}(\mathbf{s})$ (resp. $0\leq p_i\leq h_{i-1}(\mathbf{s})-1$ if $s_i=L_r$ of D) with $h_0(\mathbf{s})=0$. Let $\mathcal{LH}_n$ (resp. $\mathcal{LH}^*_n$) be the set of Laguerre histories (resp. restriced Laguerre histories) of length $n$. There are several well-known bijections between $\S_n$ and $\mathcal{LH}^*_n$ and $\mathcal{LH}_{n-1}$, see~\cite{DV94, CSZ}.
\subsection{Fran\c con-Viennot bijection}
We recall a version of Fran\c con and Viennot's bijection
$\psi_{FV}:\S_{n+1}\rightarrow\mathcal{LH}_n$. Given $\sigma\in\S_{n+1}$, the Laguerre history $\psi_{FV}(\sigma)=(\mathbf{s},\mathbf{p})$ is defined as follows:
\begin{align}
s_i=\begin{cases}
\textrm{U}&\textrm{if}\;i\in\mathrm{Val}\,\sigma\\
\textrm{D}&\textrm{if}\;i\in\mathrm{Pk}\,\sigma\\
\textrm{L}_b&\textrm{if}\;i\in\mathrm{Da}\,\sigma\\
\textrm{L}_r&\textrm{if}\;i\in\mathrm{Dd}\,\sigma
\end{cases}
\end{align}
and $p_i=(\bac)_i\,\sigma$ for $i=1,\ldots,n.$
\subsection{Restricted Fran\c con-Viennot bijection}
We recall a restricted version of Fran\c con and Viennot's bijection $\phi_{FV}: \S_n\rightarrow\mathcal{LH}_n^*.$ Given $\sigma\in\S_n$, the Laguerre history $(\mathbf{s},\mathbf{p})$ is defined as follows:
\begin{align}
s_i=\begin{cases}
\textrm{U}&\textrm{if}\;i\in\mathrm{Lval}\,\sigma\\
\textrm{D}&\textrm{if}\;i\in\mathrm{Lpk}\,\sigma\\
\textrm{L}_b&\textrm{if}\;i\in\mathrm{Lda}\,\sigma\\
\textrm{L}_r&\textrm{if}\;i\in\mathrm{Ldd}\,\sigma
\end{cases}
\end{align}
and $p_i=(\bca)_i\,\sigma$ for $i=1,\ldots,n.$
\section{Proof of Theorem~\ref{Apqcatalan}}\label{Sec:proofA}
A \emph{Motzkin path}  of length  $n$ is a sequence of points 
$\omega:=(\omega_0, \ldots, \omega_n)$  in the integer plane 
$\mathbb{Z}\times\mathbb{Z}$  such that 
\begin{itemize}
\item $\omega_0=(0,0)$ and $\omega_n=(n,0)$,
\item $\omega_i-\omega_{i-1}\in \{(1,0), (1, 1), (1,-1)\}$,
\item $\omega_i:=(x_i, y_i)\in \mathbb{N}\times \mathbb{N}$ for $i=0,\ldots, n$.
\end{itemize}
In other words, a Motzkin path of length $n\geq 0$ is a lattice path in the 
right quadrant $\mathbb{N}\times \mathbb{N}$ starting at $(0,0)$ and  ending at $(n,0)$, each step $s_i=\omega_i-\omega_{i-1}$ is a 
rise $\mathsf{U}=(1,1)$, fall
$\mathsf{D}=(1,-1)$ or level $\mathsf{L}=(1,0)$. A 2-Motzkin path is a Motzkin path with two types of level-steps $\mathsf{L}_b$ and $\mathsf{L}_r$.
Let $\mathcal{MP}_n$ be the set of 2-Motzkin paths of 
length $n$.
Clearly we can  identify 
 2-Motzkin paths of length $n$ with words $w$ on 
 $\{\mathsf{U, L_b, L_r, D}\}$ of length $n$ such that all prefixes of $w$ 
 contain no more $\mathsf{D}$'s than $\mathsf{U}$'s and the number of $\mathsf{D}$'s equals the number of $\mathsf{D}'s$.
The height of a step $s_i$ is the ordinate of the starting point $\omega_{i-1}$.

Let  $\mathbf{a}=(a_i)_{i\geq 0}$,
$\mathbf{b}=(b_i)_{i\geq 0}$, $\mathbf{b'}=(b'_i)_{i\geq 0}$  and $\mathbf{c}=(c_i)_{i\geq 1}$  be  four  sequences of indeterminates; we will work in the ring $\mathbb{Z}[[\mathbf{a}, \mathbf{b}, \mathbf{b'}, \mathbf{c}]]$. 
To each Motzkin path $\omega$ we assign a weight $W(\omega)$ that is the product of the weights for the individual steps, where 
an up-step (resp. down-step) at height 
$i$ gets weight  $a_i$ (resp. $c_i$), and 
a level-step of $\mathsf{L}_b$ (resp. $\mathsf{L}_r$) at height $i$ gets weight 
$b_i$ (resp.  $b'_i$). 
The following  result of 
 Flajolet~\cite{Fla80} is well-known.
\begin{lemma}[Flajolet]\label{flajolet} We have 
$$
\sum_{n=0}^{\infty}\left(\sum_{\omega\in \mathcal{MP}_n}W(\omega)\right)x^n
=\cfrac{1}{1-(b_0+b_0')x-\cfrac{a_0c_1x^2}{1-(b_1+b_1')x-\cfrac{a_1c_2x^2}{1-(b_2+b_2')x-\cdots}}}.
$$
\end{lemma}

For $\sigma\in\S_n(321)$, we append $\sigma(0)=0,\,\sigma(n+1)=n+1$. And 
 we also introduce a generalized $(\bac)$ statistic $\widehat{(\bac)}$ on $\sigma$ by 
$$
\widehat{(\bac)}_i\sigma:=\#\{j: i<j\leq n\; \text{and} \;\sigma(j)<\sigma(i)<\sigma(j+1)\}
$$
and 
$$
\widehat{(\bac)}=\sum_{i=1}^n\widehat{(\bac)}_i.
$$

Let
\[
C_n(p, q, t, u, w):=\sum_{\pi\in\S_n(321)}p^{\widehat{(\bac)}\,\pi}q^{(\cab)\,\pi}t^{\des\,\pi}u^{\lda\,\pi}w^{\lval\,\pi},
\]
then we have the following Theorem.

\begin{theorem}\label{thm:A1}
We have
\begin{align}
&\sum_{n\geq0}
C_n(p,q,t,u,w)z^n=\\
&\hspace{0.5cm}\cfrac{1}{1-u\cdot z-\cfrac{wtp\cdot z^{2}}{1-(p+q)u\cdot z-\cfrac{
 wtp^2q\cdot z^{2}}{1-(p^2+q^2)u\cdot z-\cfrac{wtp^3q^2\cdot z^2}{1-(p^3+q^3)u\cdot z-\cfrac{wtp^4q^3\cdot z^2}{1-\cdots}}}}}
\end{align}
\end{theorem}
It is easy to see that Theorem~\ref{thm:A1} is a generalization of Theorem~\ref{Apqcatalan}.
In this section we will give a proof for Theorem~\ref{thm:A1}.

%
%
\begin{Def}\label{def1}
A path of diagram of type A of length $n$ is a pair $(\omega',\xi')$, where $\omega'$ is a Motzkin path of length $n$ with conditions that at height $0$, the type of steps can only be $\mathrm{U}$ or $\mathrm{L}_b$ and in other heights, the type of steps can only be $\mathrm{D}$, $\mathrm{L}_b$ and $\mathrm{U}$. $\xi'=(\xi'_1,\xi'_2,\ldots,\xi'_n)$ is an integer sequence satisfying that if the $k$-th step of $\omega'$ is at height $h$ and the step is $\mathrm{D}$, then $\xi_k'=h-1$. If the $k$-th step at height $h$ is $\mathrm{L}_b$ then $\xi_k'=0$ or $\xi_k'=h$ if ($h\geq1$). For the other cases, $\xi_k'=0$.
We denote by $\mathcal{PDA}_n$ the set of path diagrams of type A of length $n$.
\end{Def}

\begin{remark}
Following Proposition~\ref{P2},  each $\sigma\in\S_n(321)$ can be constructed by inserting element one by one from 1 to $n$, then we obtain the following observations for $(\omega',\xi')=\phi_{FV}(\sigma)$.

\begin{enumerate}
\item When the $k$-th step of $\omega'$ is at height $h=0$, then 
$$\widehat{(\bac)}_k\sigma={(\cab)}_k\sigma =0.$$

\item When the $k$-th step of $\omega'$ is $\mathrm{D}$ step at height $h\geq 1$, then 
$$\widehat{(\bac)}_k\sigma=h \,\,\text{and}\,\,{(\cab)}_k\sigma =0.$$

\item When the $k$-th step of $\omega'$ is $\mathrm{L}_b$ step at height $h\geq 1$, then 
$$\widehat{(\bac)}_k\sigma=\xi_k' \,\,\text{and}\,\,{(\cab)}_k\sigma =h-\xi_k'.$$

\item When the $k$-th step of $\omega'$ is $\mathrm{U}$ step at height $h\geq 1$, then 
$$\widehat{(\bac)}_k\sigma=0 \,\,\text{and}\,\,{(\cab)}_k\sigma =h.$$

\end{enumerate}
\end{remark}

If we restrict  $\phi_{FV}$ on $\S_n(321)$ and denote as $\Phi_1$ for the restriction, then we have the following result.

\begin{lemma}
$\Phi_1$ is a bijection from $\S_n(321)$ to $\mathcal{PDA}_n$.
\end{lemma}
\begin{proof}
By Proposition~\ref{P2} and comparing the definition of $\phi_{FV}$ with Definition~\ref{def1}, we have the result.
\end{proof}
\emph{Proof of Theorem~\ref{thm:A1}}. For $\pi\in\S_n(321)$, we derive that
\begin{align}
p^{\widehat{(\bac)}\,\sigma}q^{(\cab)\,\sigma}t^{\des\,\sigma}u^{\lda\,\sigma}w^{\lval\,\sigma}=\prod_{i=1}^{n}W(s_i,\xi_i')=W((\omega',\xi')),
\end{align}
with $h$ being the height of the steps $s_i$,
\begin{align}\label{weight:dh}
W(s_i,  \xi_i')
=\begin{cases}
q^h\cdot w&\; \textrm{if}\quad s_i=U;\\
p^{h}\cdot t&\; \textrm{if}\quad s_i=D;\\
p^{\xi_i'}\cdot q^{h-\xi_i'}\cdot u&\; \textrm{if}\quad s_i=L_b.
\end{cases}
\end{align}
Therefore, the corresponding polynomial and weights become
\begin{align}
C_n(p, q, t,u, w)=&\sum_{\sigma\in\S_n(321)}p^{\widehat{(\bac)}\,\sigma}q^{(\cab)\,\sigma}t^{\des\,\sigma}u^{\lda\,\sigma}w^{\lval\,\sigma}\\
=&\sum_{(\omega',\,\xi')\in\mathcal{PDA}_{n}}W((\omega',\xi')).
\end{align}
By Lemma~\ref{flajolet}, we derive the corresponding continued fraction for the generating function of $C_n(p,q, t,u,w)$.


\section{Proof of Theorem~\ref{Bpqcatalan}}\label{Sec:proofB}
We give the following definition. Let
\[
\bB_n(p, q, t, u, v, w):=\sum_{\sigma\in\S_{n+1}(3124, 4123, 3142, 4132)}p^{(\bac)\,\sigma}q^{(\cab)\,\sigma}t^{\des\,\sigma}u^{\da\,\sigma}v^{\dd\,\sigma}w^{\val\,\sigma}.
\]

\begin{theorem}\label{thm:B1}
\begin{align}\label{cf-pqNaraB}
\sum_{n=0}^{\infty}\bB_n(p, q, t,u, v, w)z^{n}=\cfrac{1}{1-(u+tv)z-\cfrac{(p+q)twz^2}{1-p(u+tv)z-\cfrac{p^3twz^2}{1-p^2(u+tv)z-\cfrac{p^5twz^2}{\ddots}}}}.
\end{align}

\end{theorem}
It is easy to see that Theorem~\ref{thm:B1} is a generalization of Theorem~\ref{Bpqcatalan}.
In this section we will give a proof for Theorem~\ref{thm:B1}.

\begin{Def}\label{D3}
A  path diagram of type B of length $n$ is a pair $(\omega'',\,\xi'')$, where $\omega''$ is a Motzkin path of length $n$, $\xi''=(\xi''_1,\ldots,\xi''_n)$ is an integer sequence satisfying that if the $k$-th step of $\omega''$ is $\mathrm{D}$ at height $1$, then $\xi''_k=0$ or $\xi''_k=1$, otherwise $\xi''_k=h$, where $h$ is the height of the $k$-th step of $\omega''$.
We denote by $\mathcal{PDB}_n$ the set of path diagrams of type B of length $n$.
\end{Def}

\begin{remark}
Following Proposition~\ref{P3},  each $\sigma\in\S_n(3124, 4123, 3142, 4132)$ can be constructed by inserting element one by one from 1 to $n$, then we obtain that if the $k$-th step of $\omega''$ is $\mathrm{D}$ step at height $h= 1$, then  $(\cab)_k\,\sigma =1-\xi''_k$, otherwise $(\cab)_k\,\sigma =0$.




\end{remark}


If we  restrict $\psi_{FV}$ on $\S_n(3124,4123,3142,4132)$ and denote as $\Phi_2$ for the restriction, then we have the following result.
\begin{lemma}
$\Phi_2$ is a bijection from $\S_{n+1}(3124,4123,3142,4132)$ to $\mathcal{PDB}_n$.
\end{lemma}
\begin{proof}
By Proposition~\ref{P3} and comparing the definition of $\phi_{FV}$ and Definition~\ref{D3}, we have the result.
\end{proof}
\emph{Proof of Theorem~\ref{thm:B1}.} For $\sigma\in\S_n(3124,4123,3142,4132)$, we derive that
\begin{align}
p^{(\bac)\,\sigma}q^{(\cab)\,\sigma}t^{\des\,\sigma}u^{\da\,\sigma}v^{\dd\,\sigma}w^{\val\,\sigma}=\prod_{i=1}^{n}W(s_i,\xi_i'')=W((\omega'',\xi'')),
\end{align}
with $h$ being the height of the steps $s_i$,
\begin{align}
W(s_i,  \xi_i'')
=\begin{cases}
p^{\xi_i''}\cdot w\cdot t&\; \textrm{if}\quad s_i=U;\\
p^{\xi_i''}\cdot q^{h-\xi_i''}&\; \textrm{if}\quad s_i=D;\\
p^{\xi_i''}\cdot u&\; \textrm{if}\quad s_i=L_b;\\
p^{\xi_i''}\cdot v\cdot t&\; \textrm{if}\quad s_i=L_r.
\end{cases}
\end{align}
Therefore, the corresponding polynomial and weights become
\begin{align}
\bB_{n}(p,q,t,u,v,w)=&\sum_{\sigma\in\S_{n+1}(3124,4123,3142,4132)}p^{(\bac)\,\sigma}q^{(\cab)\,\sigma}t^{\des\,\sigma}u^{\da\,\sigma}v^{\dd\,\sigma}w^{\val\,\sigma}\\
=&\sum_{(\omega'',\, \xi'')\in\mathcal{PDB}_{n}}W((\omega'',\xi'')).
\end{align}
By Lemma~\ref{flajolet}, we derive the corresponding continued fraction for the generating function of $\bB_n(p,q,t,u,v,w)$.

\section{Proof of Theorem~\ref{thm:des-qgamma}}\label{Sec:proofg}

\begin{theorem}\label{thm:Aexpan} We have 
\begin{equation}\label{eq:Aexpan}
\bB_{n}(p,q,t,u,v,w)= \sum_{k=0}^{\lfloor n/2\rfloor}\gamma_{n+1,k}(p,q) (tw)^{k} (u+tv)^{n-2k}.
\end{equation}

\end{theorem}
It is easy to see that Theorem~\ref{thm:Aexpan} is a generalization of Theorem~\ref{thm:des-qgamma}.
In this section we will give two different proofs for Theorem~\ref{thm:Aexpan}.
\subsection{Continued fraction proof}
\begin{proof}[Proof of Theorem \ref{thm:Aexpan}]
Let 
\begin{align}
\bB_n(p,q,t,u,v,w)=\sum_{k=0}^{\lfloor\frac{n}{2}\rfloor}a_{n+1,k}(p,q,t,u,v)w^k,
\end{align}
so 
\begin{align}\label{H1}
a_{n+1,k}(p,q,t,u,v)=\sum_{\sigma\in\widetilde{\S}_{n+1,k}(3124, 4123, 3142, 4132)}p^{(\bac)\,\sigma}q^{(\cab)\,\sigma}t^{\des\,\sigma}u^{\da\,\sigma}v^{\dd\,\sigma},
\end{align}
where $\widetilde{\S}_{n+1,k}(3124, 4123, 3142, 4132)=\{\sigma\in \S_n(3124, 4123, 3142, 4132): \val\,\sigma=k\}$.
\\
Substituting $z\leftarrow\frac{z}{u+tv}$, $w\leftarrow\frac{w(u+tv)^2}{t}$ in~\eqref{thm:B1}, we obtain
\begin{align}\label{Q1}
&\sum_{n=0}^{\infty}\bB_n(p,\, q,\, t,\,u,\,v,\,\frac{w(u+tv)^2}{t})\cdot\frac{z^{n}}{(u+tv)^n}\\
=&\sum_{n=0}^{\infty}\sum_{k=0}^{\lfloor\frac{n}{2}\rfloor}a_{n+1,k}(p,q,t,u,v)\frac{w^k(u+tv)^{2k}}{t^k}\cdot\frac{z^n}{(u+tv)^n}\\\label{Q3}
=&\sum_{n=0}^{\infty}\sum_{k=0}^{\lfloor\frac{n}{2}\rfloor}\frac{a_{n+1,k}(p,q,t,u,v)}{t^k(u+tv)^{n-2k}}\cdot w^kz^n\\\label{Q2}
=&\cfrac{1}{1-z-\cfrac{(p+q)wz^2}{1-pz-\cfrac{p^3wz^2}{1-p^2z-\cfrac{p^5wz^2}{\ddots}}}}.
\end{align}
Since \eqref{Q2} is free of variables $t$, $u$ and $v$. The coefficient of $w^kz^n$ in  \eqref{Q3} is a polynomial in $p$ and $q$ with nonnegative integral coefficients, if we denote this coefficient by 
$$
P_{n+1,k}(p,q):=\frac{a_{n+1,k}(p,q,t,u,v)}{t^k(u+tv)^{n-2k}},
$$
then we have
\begin{align}
P_{n+1,k}(p,q)=a_{n+1,k}(p,q,1,1,0),
\end{align}
thus
\begin{align}\label{H2}
a_{n+1,k}(p,q,t,u,v)=a_{n+1,k}(p,q,1,1,0)t^k(u+tv)^{n-2k}.
\end{align}
Compare with \eqref{H1}, we have $a_{n+1,k}(p,q,1,1,0)=\sum_{\sigma\in\S_{n+1,k}}p^{(\bac)\,\sigma}q^{(\cab)\,\sigma}$, with \eqref{H2}, which proved Theorem \ref{thm:Aexpan}.


\end{proof}

\subsection{Group action proof}
Besides the patterns mentioned in the introduction, we shall also consider the so-called vincular patterns \cite{BS20}. The number of occurrences of vincular patterns $\cabd, \cadb, \dabc$, and $\dacb$  are defined by

\begin{align}
(\cabd)\,\sigma&=\#\{(i,j,k): i+1<j<k\leq n\,\, \text{and}\,\, \sigma(i+1)<\sigma(j)<\sigma(i)<\sigma(k)\},\label{def:3124}\\
(\cadb)\,\sigma&=\#\{(i,j,k): i+1<j<k\leq n\,\, \text{and}\,\, \sigma(i+1)<\sigma(k)<\sigma(i)<\sigma(j)\},\label{def:3142}\\
(\dabc)\,\sigma&=\#\{(i,j,k): i+1<j<k\leq n\,\, \text{and}\,\, \sigma(i+1)<\sigma(j)<\sigma(k)<\sigma(i)\},\label{def:4123}\\
(\dacb)\,\sigma&=\#\{(i,j,k): i+1<j<k\leq n\,\, \text{and}\,\, \sigma(i+1)<\sigma(k)<\sigma(j)<\sigma(i)\}.\label{def:4132}
\end{align}
For convenience, we also denote
\begin{align*}
(\cabd+\cadb+\dabc+\dacb)\,\sigma&:=(\cabd)\,\sigma+(\cadb)\,\sigma+(\dabc)\,\sigma+(\dacb)\,\sigma.
\end{align*}

\begin{lemma}\label{lem:pattern}
For $n\geq 1$, we have

\begin{equation}
\S_n(3124, 3142, 4123,4132)=\S_n(\cabd, \cadb, \dabc, \dacb).
\end{equation}
\end{lemma}
\begin{proof}
It is easy to see $\S_n(3124, 3142, 4123, 4132)\subseteq\S_n(\cabd, \cadb, \dabc, \dacb))$. We only need to show $\S_n(\cabd, \cadb, \dabc, \dacb))\subseteq\S_n(3124, 3142, 4123, 4132)$.\\
\begin{enumerate}
\item Suppose $\sigma\in\S_n(\cabd, \cadb, \dabc, \dacb)$ contain the pattern $3124$, i.e.,  there exists a subsequence $i<j<k<l$ such that $\sigma(l)>\sigma(i)>\sigma(k)>\sigma(j)$. It is easy to see there exist $(i',i'+1)$ such that $\sigma(i')>\sigma(k)>\sigma(i'+1)$ and $i<i'<i'+1<j$, where $i'$ and $i'+1$ might be $i$ and $j$, respectively. 
\begin{enumerate}[(a).]
\item if $\sigma(i')<\sigma(l)$, it is easy to see $\sigma(l)>\sigma(i)>\sigma(k)>\sigma(i'+1)$ for $i'<i'+1<k<l$, which is in contradiction with $\sigma\in\S_n(\cabd)$.
\item if $\sigma(i')>\sigma(l)$, it is easy to see $\sigma(i')>\sigma(l)>\sigma(k)>\sigma(i'+1)$ for $i'<i'+1<k<l$, which is in contradiction with $\sigma\in\S_n(\dabc)$.
\end{enumerate}
Above all, we have $\sigma\in\S_n(3124)$ for  $\sigma\in\S_n(\cabd, \cadb, \dabc, \dacb)$. The proof of $\sigma\in\S_n(4123)$ for  $\sigma\in\S_n(\cabd, \cadb, \dabc, \dacb)$ is similar and omitted.

\item Suppose $\sigma\in\S_n(\cabd, \cadb, \dabc, \dacb)$ contain the pattern $3142$, i.e.,  there exists a subsequence $i<j<k<l$ such that $\sigma(k)>\sigma(i)>\sigma(l)>\sigma(j)$. It is easy to see there exist $(i',i'+1)$ such that $\sigma(i')>\sigma(l)>\sigma(i'+1)$ and $i<i'<i'+1<j$, where $i'$ and $i'+1$ might be $i$ and $j$, respectively. 
\begin{enumerate}[(a).]
\item if $\sigma(i')<\sigma(k)$, it is easy to see $\sigma(k)>\sigma(i')>\sigma(l)>\sigma(i'+1)$ for $i'<i'+1<k<l$, which is in contradiction with $\sigma\in\S_n(\cadb)$.
\item if $\sigma(i')>\sigma(k)$, it is easy to see $\sigma(i')>\sigma(k)>\sigma(l)>\sigma(i'+1)$ for $i'<i'+1<k<l$, which is in contradiction with $\sigma\in\S_n(\dacb)$.
\end{enumerate}
Above all, we have $\sigma\in\S_n(3142)$ for  $\sigma\in\S_n(\cabd, \cadb, \dabc, \dacb)$. The proof of $\sigma\in\S_n(4132)$ for  $\sigma\in\S_n(\cabd, \cadb, \dabc, \dacb)$ is similar and omitted.
\end{enumerate}
Then for $\sigma\in\S_n(\cabd, \cadb, \dabc, \dacb)$, we have 
$\sigma\in\S_n(3124, 3142, 4123, 4132),$
which completes the proof.
\end{proof}

%
%
%

The main tool to  proof  \eqref{eq:Aexpan} in this Section, is  the  well-known  \textit{modified Foata-Strehl action}. See 
Foata and Strehl
\cite{FS74},  Shapiro, Woan, and Getu \cite{SWG83}
and  Br\"and\'en \cite{Bra08}.
First, we recall some definitions from \cite{FS74}.
Let $\sigma\in\S_n$ with boundary condition $\sigma(0)=0$ and $\sigma(n+1)=0$, for $x\in[n]$ the {\em$x$-factorization} of $\sigma$ is defined by
\begin{align}\label{x-factorization}
 \sigma=w_1 w_2x w_3 w_4,
\end{align}
 where $w_2$ (resp.~$w_3$) is the maximal contiguous subword immediately to the left (resp.~right) of $x$ whose letters are all larger than $x$. 
 Note that $w_1,\ldots, w_4$ may be empty. For instance,   
 if $x$ is a double ascent (resp. double descent), then $w_2=\varnothing$ (resp. $w_3=\varnothing$), and if $x$ is a peak then
$w_2=w_3=\varnothing$.
Foata and Strehl~\cite{FS74} considered a mapping $\varphi_x$, i.e., \emph{Foata-Strehl action} on permutations by
exchanging $w_2$ and $w_3$ in \eqref{x-factorization}:
$$
\varphi_x(\sigma)=w_1 w_3x w_2 w_4.
$$
For instance, if $x=3$ and $\sigma=472589316\in\S_9$, then $w_1=472,w_2=589,w_3=\emptyset$ and $w_4=16$.
Thus $\varphi_3(\sigma)=472358916$.
It is known (see \cite{FS74}) that  $\varphi_x$ is an involution acting on $\S_n$ and  that $\varphi_x$ and $\varphi_y$ commute for all $x,y\in[n]$. Br\"and\'en~\cite{Bra08} introduced the modified mapping  $\varphi'_x$ i.e., \emph{modified Foata-Strehl action (or MFS-action for short)} by
\begin{align*}
\varphi'_x(\sigma):=
\begin{cases}
\varphi_x(\sigma),&\text{if $x$ is not a valley  of $\sigma$ 
};\\
\sigma,& \text{if $x$ is a valley  of $\sigma$.}
\end{cases}
\end{align*}
Note that the boundary condition matters, e.g.,  in the above example, 
if $\sigma(n+1)=10$ instead, then $6$ becomes a double ascent and will be not fixed by $\varphi'_6$.
Also, we have $\varphi'_{x}(\sigma)=\sigma$ if $x$ is a peak, valley, otherwise $\varphi'_{x}(\sigma)$ exchanges $w_2$ and $w_3$ in the $x$-factorization of $\sigma$, which is equivalent to moving $x$ from a double ascent to a double descent or vice versa. Then $\varphi'_{x}$'s are involutions and commute.  Hence,  
for any subset $S\subseteq[n]$ we can define the map $\varphi'_S :\S_n\rightarrow\S_n$ by
\begin{align*}
\varphi'_S(\sigma)=\prod_{x\in S}\varphi'_x(\sigma).
\end{align*}
In other words,  the group $\Z_2^n$ acts on $\S_n$ via the mapping $\varphi'_S$ with $S\subseteq[n]$. 
For example,  let  $\sigma=472589316\in \S_9$.
If  $S=\{3, 4,5\}$, we have $\varphi'_{S}(\sigma)=742389516$, 
see Fig.~\ref{valhopFmax} for an illustration.

\begin{figure}[t]
\begin{center}
\begin{tikzpicture}[scale=0.4] 	
\draw[step=1,lightgray,thin] (0,0) grid (10,10); 
	\tikzstyle{ridge}=[draw, line width=1, dotted, color=black] 
	\path[ridge] (0,0)--(1,4)--(2,7)--(3,2)--(4,5)--(5,8)--(6,9)--(7,3)--(8,1)--(9,6)--(10,0); 
	\tikzstyle{node0}=[circle, inner sep=2, fill=black] 
	\tikzstyle{node1}=[rectangle, inner sep=3, fill=mhcblue] 
	\tikzstyle{node2}=[diamond, inner sep=2, fill=davidsonred] 
	\node[node0] at (0,0) {}; 
	\node[node2] at (1,4) {}; 
	\node[node0] at (2,7) {}; 
	\node[node0] at (3,2) {}; 
	\node[node2] at (4,5) {}; 
	\node[node0] at (5,8) {}; 
	\node[node0] at (6,9) {}; 
	\node[node2] at (7,3) {}; 
	\node[node0] at (8,1) {};
	\node[node0] at (9,6){};
	\node[node0] at(10,0){}; 
	\tikzstyle{hop1}=[draw, line width = 1.5, color=davidsonred,->]
	\tikzstyle{hop2}=[draw, line width = 1.5, color=davidsonred,<-] 
	\path[hop1] (4,5)--(6.5,5);
	\path[hop2] (3.5,3)--(7.3, 3);
	\path[hop1](1,4)--(2.5,4);
	\tikzstyle{pi}=[above=-1] 
	\node[pi] at (0,0) {0}; 
	\node[pi] at (1,0) {$\red{4}$}; 
	\node[pi] at (2,0) {$7$}; 
	\node[pi] at (3,0) {2}; 
	\node[pi] at (4,0) {\red{5}}; 
	\node[pi] at (5,0) {$8$}; 
	\node[pi] at (6,0) {$9$}; 
	\node[pi] at (7,0) {\red{3}}; 
	\node[pi] at (8,0) {1};
	\node[pi] at (9,0) {6};
	\node[pi] at (10,0) {0};
	\path[draw,line width=1,->] (11,5)--(15,5); 
	\begin{scope}[shift={(16,0)}] 
	\draw[step=1,lightgray,thin] (0,0) grid (10,10); 
	\path[ridge] (0,0)--(1,7)--(2,4)--(3,2)--(4,3)--(5,8)--(6,9)--(7,5)--(8,1)--(9,6)--(10,0); 
	\node[node0] at (0,0) {}; 
	\node[node0] at (1,7) {}; 
	\node[node0] at (2,4) {}; 
	\node[node0] at (3,2) {}; 
	\node[node0] at (4,3) {}; 
	\node[node0] at (5,8) {}; 
	\node[node0] at (6,9) {}; 
	\node[node0] at (7,5) {}; 
	\node[node0] at (8,1) {}; 
	\node[node0] at (9,6){};
	\node[node0] at (10,0){};
        \node[pi] at (0,0) {0}; 
	\node[pi] at (1,0) {$7$}; 
	\node[pi] at (2,0) {$4$}; 
	\node[pi] at (3,0) {2}; 
	\node[pi] at (4,0) {3}; 
	\node[pi] at (5,0) {$8$}; 
	\node[pi] at (6,0) {$9$};
	\node[pi] at (7,0) {5}; 
	\node[pi] at (8,0) {1};
	\node[pi] at (9,0){6};
	\node[pi]  at (10,0){0};
\end{scope}
\end{tikzpicture}
\end{center}
\caption{MFS-action on $\sigma=472589316\in \S_9$ with $S=\{3, 4, 5\}$
yields $\varphi'_{S}(\sigma)=742389516$\label{valhopFmax}}
\end{figure}
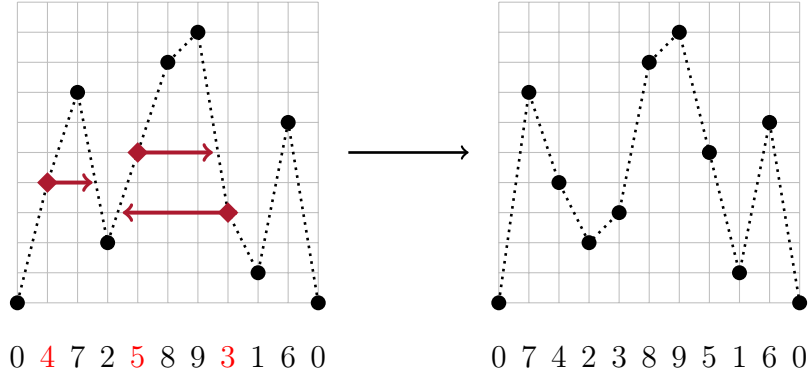


For  entries $k,\,l\in [n]$ with $k<l$, we define the following refined statistics on $\sigma\in\S_n$,
\begin{align*}
(\cabd)\{k,l\}(\sigma):=&\#\{i:1<i<\sigma^{-1}(k)< \sigma^{-1}(l)\,\, \text{and}\,\, \sigma(i)<k<\sigma(i-1)<l\},\nonumber\\
(\cadb)\{k,l\}(\sigma):=&\#\{i:1<i<\sigma^{-1}(l)< \sigma^{-1}(k)\,\, \text{and}\,\,  \sigma(i)<k<\sigma(i-1)<l\},\nonumber\\
(\dabc)\{k,l\}(\sigma):=&\#\{i:1<i<\sigma^{-1}(k)<\sigma^{-1}(l)\,\, \text{and}\,\, \sigma(i)<k<l<\sigma(i-1)\},\nonumber\\
(\dacb)\{k,l\}(\sigma):=&\#\{i:1<i<\sigma^{-1}(l)< \sigma^{-1}(k)\,\, \text{and}\,\, \sigma(i)<k<l<\sigma(i-1)\}.
\end{align*}
We denote
\begin{align}
&(\cabd+\cadb+\dabc+\dacb)\{k,l\}(\sigma)\nonumber\\
:=&(\cabd)\{k,l\}(\sigma)+(\cadb)\{k,l\}(\sigma)+(\dabc)\{k,l\}(\sigma)+(\dacb)\{k,l\}(\sigma),\nonumber
\end{align}
then the following Lemma holds.
\begin{lemma}\label{One}
For $\sigma\in\S_n$ and $k, l\in[n]$ with $k<l$, the statistic $(\cabd+\cadb+\dabc+\dacb)\{k,l\}$ is invariant under the group action $\varphi'_S$ with  $S\subset [n]$, i.e.,
\begin{align}
&(\cabd+\cadb+\dabc+\dacb)\{k,l\}(\sigma)\\
=&(\cabd+\cadb+\dabc+\dacb)\{k,l\}(\varphi'_S(\sigma)).
\end{align}
\end{lemma}
\begin{proof}
An alternative description of $(\cabd+\cadb+\dabc+\dacb)\{k,l\}(\sigma)$ is the number pairs $(\sigma(i),\sigma(j))$ such that $1\leq i<j<\min\{\sigma^{-1}(k), \sigma^{-1}(l)\}$ and there exists $i\leq m<j$ such that $\sigma(m+1)<k<\sigma(m)<l$ or $\sigma(m+1)<k<l<\sigma(m)$, where $(\sigma(i),\sigma(j))$ is a pair of consecutive peak and valley. By consecutive we mean that there are no other peaks and valleys in between $\sigma(i)$ and $\sigma(j)$. The number of such pairs is invariant under the action since $\sigma(i)$ and $\sigma(j)$ can not move and $k, l$ can not move over the peak $\sigma(i)$.
\end{proof}
Since we have 
\begin{align}\label{kk1}
(\cabd+\cadb+\dabc+\dacb)\,\sigma=\sum_{\substack{k,\, l\in[n]\\k<l}}(\cabd+\cadb+\dabc+\dacb)\{k,l\}(\sigma),
\end{align}
then we have following Lemma.

\begin{lemma}\label{lem:stat}
For $\sigma\in \S_n$ 
 the triple permutation statistic $(\pk, \val, \cabd+\cadb+\dabc+\dacb)$ is invariant under the group action $\varphi'_S$ with  $S\subset [n]$, i.e.,
 \begin{align}\label{quintuple}
 &(\pk, \val, \cabd+\cadb+\dabc+\dacb)\;\sigma\nonumber\\
 =&(\pk, \val,  \cabd+\cadb+\dabc+\dacb)\;\varphi'_S(\sigma).
\end{align}
\end{lemma}
\begin{proof}
Since the statistics $\pk$ and $\val$ are invariant under  MFS-action, and Lemma~\ref{One} with \eqref{kk1} ensure that $(\cabd+\cadb+\dabc+\dacb)$ is invariant under MFS-action, which proves the Lemma.

\end{proof}

\begin{figure}[htb]
\setlength {\unitlength} {0.8mm}
\begin{picture} (90,45) \setlength {\unitlength} {1mm}
\thinlines




\put(-37, -2){$0$}
\put(-36, 2){\circle*{1.3}}
\put(-11,27){\line(-1,-1){25}}
\put(-11,27){\circle*{1.3}}\put(-13, 29){$\sigma(i)$}
\put(-4, 20){\circle*{1.3}}
\put(13, 3){\circle*{1.3}}
\put(13,3){\line(-1,1){24}}

\put(10,-2){$\sigma(j)$}
\put(45,35){\circle*{1.3}}
\put(45,35){\line(1,-1){15}}\put(60,20){\circle*{1.3}}
\put(13,3){\line(1,1){32}}
\put(68, 28){\circle*{1.3}}
\put(60,20){\line(1,1){8}}
\put(94,2){\circle*{1.3}}
\put(68,28){\line(1,-1){26}}

\put(-4,20){\circle*{1.3}}\put(-18,20){\dashline{1}(1,0)(10,0)}\put(-18,20){\vector(-1,0){0.1}}
\put(-2,20){$\sigma(m)$}
\put(22,12){\circle*{1.3}}\put(22,12){\dashline{1}(1,0)(62,0)}\put(84,12){\vector(1,0){0.1}}
\put(18,12){$k$}
\put(50,30){\circle*{1.3}}\put(40,30){\dashline{1}(1,0)(10,0)}\put(40,30){\vector(-1,0){0.1}}
\put(52,30){$l$}

\put(94,-2){$0$}
\end{picture}
\caption{Valley-hopping on $(\cabd+\cadb+\dabc+\dacb)\{k, l\}(\sigma)$
\label{Fig:3124}}
\end{figure}

 Note that in Fig.~\ref{valhopFmax}, we have 
 $$(\cabd+\cadb+\dabc+\dacb)\,\sigma=9=(\cabd+\cadb+\dabc+\dacb)\,\varphi'_S(\sigma).$$

 \begin{lemma}
 The MFS-action $\varphi'_S$ is closed on the subset $\S_n(3124, 3142, 4123, 4132)$.
 
 \end{lemma}
 \begin{proof}
 This follows directly form Lemmas~\ref{lem:pattern} and \ref{lem:stat}.
 
 \end{proof}

 \begin{lemma}\cite[Theorem~3.8]{FTHZ}\label{lem:stat2}
 The statistics $(\bac)$ and $(\cab)$ are constant on any orbit under the MFS-action.
 \end{lemma}
 \vskip 3mm
 
 \begin{proof}[Proof of Theorem~\ref{thm:Aexpan}]
 For any  permutation $\sigma\in \S_n$,  let $\Orb(\sigma)=\{g(\sigma): g\in\Z_2^n\}$ be the orbit of $\sigma$ under the MFS-action. 
The MFS-action divides the set $\S_n$ into disjoint orbits. 
 Moreover, for $\sigma\in \S_n$, if $x$ is a double descent of $\sigma$, then 
 $x$ is a double ascent of $\varphi'_x(\sigma)$.
 Hence, there is a unique  
  permutation in each orbit which has no 
 double descent.
 Now, let 
  $\bar{\sigma}$ be such a unique element  in $\Orb(\sigma)$, then 
  \begin{align*}
  \da(\bar{\sigma})&=n-\pk(\bar{\sigma})-\val(\bar{\sigma});\\
   \des(\bar{\sigma})&=\pk(\bar{\sigma})-1=\val(\bar{\sigma}).
   \end{align*}
 And for any other $\sigma'\in\Orb(\pi)$, it can be obtained from $\bar{\pi}$ by repeatedly applying $\varphi'_x$ for some double ascent $x$ of $\bar{\sigma}$. Each time this happens, $\des$ and $\dd$ increases by $1$, $\da$ decreases by $1$, and $\pk$ and $\val$ keep unchanged.
We have 
\begin{align*}\label{orbit}
&\sum_{\sigma'\in\Orb(\sigma)}t^{\des\,\sigma'}u^{\da\,\sigma'}v^{\dd\,\sigma'}w^{\valley\,\sigma'}\\
=&(tw)^{\valley\,\bar{\sigma}}(u+tv)^{\da\,\bar{\sigma}}\\
=&(tw)^{\des\,\bar{\sigma}}(u+tv)^{n-2\des\,\bar{\sigma} -1}.
\end{align*}
Therefore, by Lemma~\ref{lem:stat2}, we obtain \eqref{eq:Aexpan}.

 \end{proof}
 
 \section{Another kind of $(q,t)$-Catalan numbers}\label{Varicatalan}
 In this section, we introduce another kind of $(q, t)$-Catalan numbers.
 
 Let
 \begin{equation}\label{p2}
\widehat{C}_n(q,t,u,v,w):=\sum_{\sigma\in\S_{n+1}(312)}q^{(\bac)\,\sigma}t^{\des\,\sigma}u^{\da\,\sigma}v^{\dd\,\sigma}w^{\val\,\sigma}
\end{equation}
 \begin{theorem}\label{T1}
We have
\begin{align}
&\sum_{n\geq0}\widehat{C}_n(q,t,u,v,w)z^{n}=\nonumber\\
&\hspace{2cm}\cfrac{1}{1-(u+vt)\cdot z-\cfrac{wtq\cdot z^{2}}{1-q(u+vt)\cdot z-\cfrac{
 wtq^3\cdot z^{2}}{1-q^2(u+vt)\cdot z-\cfrac{wtq^5\cdot z^2}{1-\cdots}}}}\label{eq:c2}
\end{align}
\end{theorem}

\begin{remark}
When $(q,t,u,v,w)=(q,t,1,1,1)$, Eq~\eqref{eq:c2} reduces to the third interpretation of \cite[Theorem 1]{FTHZ}.
The above theorem could be derived by \cite[Eq.~(28)]{SZ12} and \cite[Lemma~2.8]{FTHZ}, which was proved by the restricted version of continued fraction on $\S_n$. However, we compute the continued fraction by using {\em insertion encoding} and Fran\c con-Viennot's bijection on $\S_n(321)$. 
\end{remark}



 \begin{Def}\label{def3}
A  path diagram of type C of length $n$ is a pair $(\omega,\,\xi)$, where $\omega$ is a Motzkin path of length $n$, $\xi=(\xi_1,\ldots,\xi_n)$ is integer sequence satisfying that if the $k$-th step of $\omega$ is at height $h$, then $\xi_k=h$.
We denote by $\mathcal{PDC}_n$ the set of path diagrams of type C of length $n$.
\end{Def}

%
%
%
%
%
%
%
%


For $\sigma\in\S_n(312)$, we construct the path diagram $\Phi_3(\sigma)=(\omega,\xi)$, where $\omega=(\omega_0,\omega_1,\ldots,\omega_n)$ is defined as follows. For $j\in[n]$, $s_j=\omega_j-\omega_{j-1}$,
\begin{align}
s_j=\begin{cases}
\textrm{U}&\textrm{if}\;j\in\mathrm{Val}\,\sigma\\
\textrm{D}&\textrm{if}\;j\in\mathrm{Pk}\,\sigma\\
\textrm{L}_b&\textrm{if}\;j\in\mathrm{Da}\,\sigma\\
\textrm{L}_r&\textrm{if}\;j\in\mathrm{Dd}\,\sigma
\end{cases}.
\end{align}
As we have $\omega$, then $\xi$ can be obtained by the height of each step of $\omega$.
\begin{lemma}
The mapping $\Phi_3: \sigma\mapsto(\omega,\xi)$ is a bijection. 
\end{lemma}
\begin{proof}
By Proposition~\ref{P1}, we can see that $\Phi_3$ is actually a restriction of Fran\c con-Viennot bijection on $\S_n(312)$.
\end{proof}
\emph{Proof of Theorem~\ref{T1}}. For $\sigma\in\S_n(312)$, we derive that 
\begin{align}
q^{(\bac)\,\sigma}t^{\des\,\sigma}u^{\da\,\sigma}v^{\dd\,\sigma}w^{\val\,\sigma}=\prod_{i=1}^{n+1}W(s_i,\xi_i)=W((\omega,\xi)),
\end{align}
with $h$ being the height of the steps $s_i$,
\begin{align}\label{weight:dh}
W(s_i,  \xi_i)
=\begin{cases}
q^{\xi_i}\cdot w&\; \textrm{if}\quad s_i=U;\\
q^{\xi_i}\cdot t&\; \textrm{if}\quad s_i=D;\\
q^{\xi_i}\cdot u&\; \textrm{if}\quad s_i=L_b;\\
q^{\xi_i}\cdot v\cdot t&\; \textrm{if}\quad s_i=L_r.
\end{cases}
\end{align}
Therefore, the corresponding polynomial and weights become
\begin{align}
\widehat{C}_n(q,t,u,v,w)=&\sum_{\pi\in\S_{n+1}(312)}q^{(\bac)\,\pi}t^{\des\,\pi}u^{\da\,\pi}v^{\dd\,\pi}w^{\val\,\pi}\\
=&\sum_{(\omega,\,\xi)\in\mathcal{PDC}_{n}}W((\omega,\xi)).
\end{align}
By Lemma~\ref{flajolet}, we derive the corresponding continued fraction for the generating function of $\widehat{C}_n(q,t,u,v,w)$.

\end{document}